\documentclass[12pt,a4paper]{amsart}
\usepackage{amsmath}
\usepackage[top=35mm, bottom=30mm, left=30mm, right=30mm]{geometry}
\usepackage{mathptmx}
\usepackage{mathrsfs}
\usepackage{verbatim}
\usepackage{url}
\usepackage[all]{xy}
\usepackage{xcolor}
\usepackage[colorlinks=true,citecolor=blue]{hyperref}
\usepackage{amsmath,amsthm}
\usepackage{amssymb}
\usepackage{enumerate}
\usepackage{graphicx}

\newtheorem{thm}{Theorem}[section]
\newtheorem{cor}[thm]{Corollary}

\newtheorem{lem}[thm]{Lemma}
\newtheorem{prop}[thm]{Proposition}

\newtheorem{ques}[thm]{Question}

\theoremstyle{definition}
\newtheorem{defin}[thm]{Definition}
\newtheorem{rem}[thm]{Remark}

\numberwithin{equation}{section}
\begin{document}


\baselineskip=20pt


\title{$(\omega, \alpha, n)$-sensitivity and limit sets of zero entropy homeomorphisms on the square}

\author[J.~Mai]{Jiehua Mai}
\address{School of Mathematics and Quantitative
Economics, Guangxi University of Finance and Economics, Nanning,
Guangxi, 530003, P. R. China \& Institute of Mathematics, Shantou
University, Shantou, Guangdong, 515063, P. R. China}
\email{jiehuamai@163.com; jhmai@stu.edu.cn}

\author[E.~Shi]{Enhui Shi}
\thanks{*Corresponding author}
\address{School of Mathematics and Sciences, Soochow University, Suzhou, Jiangsu 215006, China}
\email{ehshi@suda.edu.cn}

\author[K.~Yan]{Kesong Yan*}
\address{School of Mathematics and Statistics, Hainan Normal
University,  Haikou, Hainan, 571158, P. R. China}
\email{ksyan@mail.ustc.edu.cn}

\author[F.~Zeng]{Fanping Zeng}
\address{School of Mathematics and Quantitative
Economics, Guangxi University of Finance and Economics, Nanning, Guangxi, 530003, P. R. China}
\email{fpzeng@gxu.edu.cn}

\begin{abstract}
For a homeomorphism $f$ of a compact metric space $X$ and a positive integer $n\geq 2$, we introduce the notion of $(\omega, \alpha, n)$-sensitivity of $f$,
which describes such a kind of chaos: there is some $c>0$ such that for any $x\in X$ and any open neighborhood $U$ of $x$, there are points $\{x_i\}_{i=1}^n$ and $\{y_i\}_{i=1}^n$ in $U$ such that both the collection of $\omega$-limit sets $\omega(x_i, f)$ and that of the $\alpha$-limit sets $\alpha(y_i, f)$
are pairwise $c$-separated. Then we construct a class of homeomorphisms of the square $[-1, 1]^2$ which are $(\omega, \alpha, n)$-sensitive for any $n\geq 2$ and have zero topological entropies. To investigate further the complexity of zero entropy homeomorphisms by using limit sets, we analyze in depth the limit sets of square homeomorphisms  by the boundary permeating technique. Specially, we prove that for any given set of points $Y\equiv\{y_{n1}, y_{n2}:n\in\mathbb N\}$ in $(-1, 1)^2$ which satisfies some loosely technical conditions, and for any given family of pairwise disjoint countable
dense subsets $\{W_n:n\in\mathbb N\}$ of $(-1, 1)^2-Y$, there is a zero entropy homeomorphism $f$ on the square $[-1, 1]^2$ such that
$\omega(x, f)=\{y_{n1}\}$ and $\alpha(x, f)=\{y_{n2}\}$ for any $n$ and any $x\in W_n$.
\end{abstract}



\maketitle

\pagestyle{myheadings} \markboth{J. Mai, E. Shi, K. Yan and F.
Zeng}{$(\omega, \alpha, n)$-sensitivity and limit sets}

\section{Introduction}

Positive topological entropy is commonly regarded as a chaotic property. This is
reasonable as a map of positive entropy really exhibits many complex dynamical behaviors.
It is known that positive entropy implies Li-Yorke chaos\  \cite{BGKM02, KL07},
 Chaos DC2 (a variant of distributional chaos)\ \cite{Do14, HLY14}, and the existence of a weak horseshoe\ \cite{HL17}.
However, a map of zero entropy can also be chaotic if we take a different view,  even if the phase space is only
a closed interval (see  \cite{Sm86, Xi86}). Using topological sequence entropy introduced by Goodman \cite{Go74},
people studied the complexity and structure of some special systems of zero entropy \cite{FS91, HLSY03, MS07}.
In 2010, Sarnak \cite{Sa12} formulated the famous M\"obius disjointness conjecture, which establishes a deep
connection between zero entropy systems and number theories. Huang-Wang-Ye \cite{HWY19}
 proved the conjecture for any dynamical system  of sub-polynomial measure complexity.
These greatly stimulates the study of zero entropy systems. For instance, Huang-Xu-Ye \cite{HXY20}
gave a surprising construction of a minimal distal maps on the torus with sub-exponential measure complexity.
One may consult \cite{HLTXY21, LOYZ17, LOZ22} for some related work very recently.
\medskip

Sensitivity on initial conditions (or sensitivity for short) was defined by Guckenheimer \cite{Gu79},
the phrase of which was used by Ruelle in \cite{Ru78}. In fact, the idea of sensitivity can
date back to the earlier work of Poincar\'e and Lorenz. Nowadays, sensitivity has been widely
understood as a key feature of a chaotic system. In the definitions of
Devaney chaos \cite{De89} and  Auslander-Yorke chaos \cite{AY80}, sensitivity appears directly as
an ingredient. It is worth mentioning that a transitive non-periodic system having a dense set of periodic points
is sensitive \cite{BBCDS}, which has been extended to transitive non-minimal systems with dense minimal points \cite{GW93, AAB96}.
One may consult \cite{MYZ24} for a further generalization.  Some extended definitions of sensitivity, such as Li-Yorke sensitivity, $n$-sensitivity,
mean-sensitivity,  multi-sensitivity, and thick sensitivity are introduced and intensively explored
(see e.g. \cite{AK03, HKKZ16,  HKZ18, LTY15, SYZ08, Xi05}).
\medskip

We formulate the following extension of sensitivity using limit sets. The idea of using $\omega$-limit sets
to define chaos already appeared in \cite{Li93}.

\begin{defin}
Let $(X, d)$ be a compact metric space, and
$f:X\rightarrow X$  be a continuous map. For $n \geq 2$ , $f$ is said
to be $(\omega, n)$-{\it sensitive} if there exists $s>0$ such that,
for any $x\in X$ and any neighborhood $U$ of $x$, there exist
$\{y_1, \ldots, y_n\}\subset U$ such that $d(\omega(y_i, f), \omega(y_j, f))>s$
for any $i\not=j$; $f$ is said to be $\omega$-{\it sensitive} if $f$ is $(\omega , 2)$-sensitive.
If $f$ is a homeomorphism, then $f$ is said to be
$(\alpha, n)$-{\it sensitive} (resp. $\alpha$-{\it sensitive}) if $f^{-1}$
is $(\omega, n)$-sensitive (resp. $\omega$-sensitive); $f$ is said to be
$(\omega, \alpha, n)$-sensitive if $f$ is both $(\omega, n)$-sensitive and
$(\alpha, n)$-sensitive; $f$ is said to be $(\omega, \alpha)$-{\it sensitive} if
$f$ is $(\omega, \alpha, 2)$-sensitive.
\end{defin}

For brevity of statement, we call $f$ is $(\omega, \infty)$-{\it sensitive} (resp. $(\omega, \alpha, \infty)$-{\it sensitive})
if $f$ is $(\omega, n)$-sensitive (resp. $(\omega, \alpha, n)$-sensitive) for any $n\geq 2$.

\begin{rem}\label{rem}
(1) There are several facts followed immediately from the definition.
$\omega$-sensitivity is stronger than sensitivity, and $(\omega, n)$-sensitivity
is stronger than $n$-sensitivity defined by Xiong in \cite{Xi05}. If
$f$ is $(\omega, n)$-sensitive, then it has at least $n$ minimal sets; so
no minimal system is $\omega$-sensitive.

(2) It is easy to show that if $f$ is transitive, has infinitely many minimal sets, and
has the pseudo-orbit-tracing property, then it is $(\omega, \infty)$-sensitive. Thus, the maps such as
 transitive Anosov diffeomorphisms on compact manifolds, shifts on symbolic spaces $\{1,\ldots, k\}^{\mathbb Z}$ with $k\geq 2$,
 and covering maps on the circle with degree $\geq 2$ are all $(\omega,  \infty)$-sensitive. It is well known
 that all these maps have positive entropy.

(3) Following the same lines as in the proof of Proposition 2.40 in \cite{Ru18}, we can show that
if $f$ is a sensitive graph map, then there is a nondegenerate $f^n$-invariant subgraph for some $n>0$
the restriction of $f^n$ to which is transitive. Then by results in \cite{Bl87} and \cite{MSh07},
we see that no graph map of zero entropy is sensitive. Specially, no homeomorphism on a graph is sensitive;
this also holds for homeomorphisms on a locally connected continuum with a free arc \cite{MShi12} and the
locally connectedness condition is necessary \cite[Theorem 3.9]{MShi07}.

(4)  Let $A=\mathbb S^1\times [0, 1]$ be the annulus, where $\mathbb
S^1$ is the unit circle in the complex plane. Let $f:A\rightarrow A$
be defined by $$f(e^{2\pi {\rm i}x}, y)=(e^{2\pi {\rm i}(x+y)}, y),\
\ \mbox{for all\ }(x, y)\in \mathbb R\times [0, 1].$$ Then $f$ is
$n$-sensitive for any $n\geq 2$, but is not $\omega$-sensitive.
Clearly, $f$ is a distal homeomorphism and so it is of zero entropy.
\end{rem}

From (2) of Remark \ref{rem}, we see that many important positive entropy maps have $(\omega,  \infty)$-sensitivity;
(3) of the remark indicates that sensitivity does not occur for zero entropy maps on compact $1$-manifolds;
and (4) shows that there exists a zero entropy homeomorphism on a $2$-dimensional compact manifold which is $n$-sensitive for any $n$, but
is not $\omega$-sensitive. So, one may doubt whether there exists an $(\omega,  \infty)$-sensitive homeomorphism of zero
entropy on a compact $2$-dimensional manifold.
We obtain the following result, the proof of which relies on analyzing in depth the limit sets of normally rising homeomorphisms on the square
(see Section \ref{sec-3}).

\begin{thm}\label{main1}
There is a zero entropy homeomorphism on the square $[-1, 1]^2$ which is $(\omega, \alpha, \infty)$-sensitive.
\end{thm}

In order to further investigate the complexity of zero entropy homeomorphisms by using limit sets, we develop the boundary permeating technique
in Section \ref{Sec-4} and introduce the notion of permeable hanging set. We show that a permeable hanging set must be
a dendrite as a topological space, and determine the conditions under which any element of a given family of permeable hanging sets
is the limit set of a countable dense subsets coming from a pre-described family of sets (see Theorem \ref{thm:4-7} for the details). Applying the results in Section \ref{Sec-4}, we obtain the following theorem (Theorem \ref{cor:4-9}).

\begin{thm} \label{main2}
Let $Y_{\mathbb{N}}=\{y_{n1}, y_{n2}: n \in \mathbb{N}\}$ be a
countable infinite set in the open square $(-1, 1)^2$, and
$\{W_n: n \in \mathbb{N}\}$ be a family of pairwise disjoint
countable sets in $(-1, 1)^2-Y_{\mathbb{N}}$, in which each $W_n$ is dense in
$[-1, 1]^2$. Suppose that
\begin{enumerate}
\item[{\rm (1)}] $\lim_{n \rightarrow \infty} d(y_{n1}, J_1)=\lim_{n \rightarrow
\infty} d(y_{n2}, J_{-1})=0$, where $J_i=[-1, 1]\times\{i\}$\  for $i\in \{-1, 1\}$;

\medskip

\item[{\rm (2)}] For any $\beta \in p(Y_{\mathbb{N}})$, there exists
$\varepsilon_{\beta}>0$ such that $Y_{\mathbb{N}} \cap
p^{-1}([\beta-\varepsilon_{\beta}, \beta+\varepsilon_{\beta}])$ is a
finite set, where $p:\mathbb R^2\rightarrow \mathbb R$ is the projection to the
first coordinate.
\end{enumerate}
Then there exists a homeomorphism $\psi: [-1, 1]^2 \rightarrow [-1, 1]^2$ of zero entropy such
that, for any $n \in \mathbb{N}$ and any $w \in W_n$, one has
$\omega(w, \psi)=\{y_{n1}\}$ and $\alpha(w, \psi)=\{y_{n2}\}$.
\end{thm}

Theorem \ref{main2} indicates that a square homeomorphism of zero entropy can have very
complex dynamical behaviors if we take the view of limit sets. One may compare this result with
other interesting results around limits sets of triangular maps and spiral maps \cite{BM10,JS01,KM10, KS92}.
The construction process  to prove the Theorem is a mechanism that produces homeomorphisms of all kinds of complex
limit sets, which has been used in dealing with other problems around plane homeomorphisms \cite{MSYZ24-6, MSYZ24}.

\section{Definitions and notations}

Throughout the paper, we use $\mathbb Z$, $\mathbb Z_+$, $\mathbb Z_-$,
and $\mathbb N$ to denote the sets of integers, nonnegative integers,
nonpositive integers, and positive integers, respectively. For any
\;\!$n\in\mathbb{N}$\;\!, \,write \;\!$\mathbb{N}_n=\{1,\cdots ,
n\}$. For\ \ $r\,,\,s\,\in\,\mathbb R$\,,\ \,we use \ $(r,s)$ \ to
denote a point in \,$\mathbb R^{\,2}$\,.\ \ If \ $r\,<\,s$\;,\ \ we
also use \ $(r,s)$\ \ to denote an open interval in\ \,$\mathbb
R$.\,\,\,These will not lead to confusion\,.\,\,\,For example\,,
\,if we write \ $(r,s)\,\in\,X$\,, \ then \ $(r,s)$ \ will be a
point\,; \ if we write \ $t\,\in\,(r,s)$\,,\ \ then \ $(r,s)$ \ will
be a set\,,\ \,and hence is an open interval\,.

\medskip

Let $d$ be the Euclidean metric on $\mathbb R^2$. For
$x\in\mathbb R$, $X, Y\subset \mathbb R^2$, and $r\geq 0$, write
$B(x, r)=\{y\in\mathbb R^2: d(y, x)\leq r\}$, $B(X, r)=\bigcup\{B(x,
r): x\in X\}$, and $d(X, Y)=\inf\{d(x, y): x\in X, y\in Y\}$. Denote
by ${\rm diam}(X)\equiv\sup\{d(y, z):y,\ z\in X\}$ the
{\it {diameter}} of $X$. A space $Y$ is called a
{\it{disc}} (resp. a {\it {circle}}, resp. an {\it
{arc}}) if $Y$ is homeomorphic to the unit disc ${\mathbb
D}\equiv B(0, 1)$ (resp. to the unit circle $\mathbb S^1\equiv
\partial {\mathbb D}$, resp. to the unit interval $[0, 1]$). If $I$
is an arc and $h:I \rightarrow [0, 1]$ is a homeomorphism, then
$\partial I \equiv\{h^{-1}(0), h^{-1}(1)\}$ is called the {\it
{endpoint set}} of $I$; we use $\stackrel{\circ} {I}$ to
denote $I-\partial I$. Notice that we also use $\stackrel{\circ}
{Y}$ to denote the interior ${\rm Int}\ Y$ of $Y$ for any
$Y\subset\mathbb R^2$. For any circle $C$ in $\mathbb{R}^2$, denote
by $\mathrm{Dsc}(C)$ the disc in $\mathbb{R}^2$ such that $\partial
\mathrm{Dsc}(C)=C$.

\medskip

A {\it continuum} is a connected compact metric space and a {\it Peano continuum}
 is a locally connected continuum.  A {\it graph} is a continuum which can be written as the union of finitely
many arcs, any two of which are either disjoint or intersect only in
one or both of their end points. A {\it {tree}} is a graph containing no circle.
A {\it dendrite} is a Peano continuum containing no circle.  Clearly, a tree is a dendrite.

\medskip

Let $X$ be a metric space and $f:X\rightarrow X$ be a homeomorphism.
Let $f^{\;\!0}$ be the identity map of $X$, and let
$f^{\;\!n}=f\circ f^{\;\!n-1}$ be the composition map of $f$ and
$f^{\;\!n-1}$. For $x \in X$, the sets $O(x,
f)\equiv\{f^n(x):n\in\mathbb Z\}$, $O_+(x,
f)\equiv\{f^n(x):n\in\mathbb Z_+\}$, and $O_-(x,
f)\equiv\{f^{-n}(x):n\in\mathbb Z_-\}$ are called the {\it
{orbit}}, {\it {positive orbit}}, and {\it
{negative orbit}} of $x$, respectively. For $x, y\in X$, if
there is a sequence of positive integers $n_1<n_2<\cdots$ such that
$f^{n_i}(x)\rightarrow y$  then we call $y$ an {\it
{$\omega$-limit point}} of $x$. We denote by $\omega(x, f)$ or $\omega_f(x)$
the set of all $\omega$-limit points of $x$ and call it the {\it
{$\omega$-limit set}} of $x$. If $y \in \omega(x, f^{-1})$,
then we call $y$ an {\it {$\alpha$-limit point}} of $x$. We
denote by $\alpha(x, f)$ or $\alpha_f(x)$ the set of all $\alpha$-limit points of $x$
and call it the {\it {$\alpha$-limit set}} of $x$.

\section{Limit sets of normally rising homeomorphisms on the square $J^2$}\label{sec-3}

We always write\ \ $J=\,[\,-1\,,\,1\,]$\;.\ \ \,Define the
homeomorphism \ $f_{01}\,:\,J\,\to\;J$\ \ by\ , \ \,for any \
$s\,\in\,J$,
$$
f_{01}(s)\,=\;
   \left\{\begin{array}{ll}(s+1)/2\,, & \mbox{\ \ \ \ \ if \ }\;0\,\le\,s\,\le1\;; \\
                        \!\;\;s+1/2\,, & \mbox{\ \ \ \ \ if \ \ }-1/2\,\le\,s\le\,0\;;  \\
                             \;2s+1\,, & \mbox{\ \ \ \ \ if \ }-1\,\le\,s\le\,-1/2\ .
\end{array}\right.
$$
Let the homeomorphism \ $f_{02}\,:\,J^{\,2}\,\to\;J^{\,2}$\ \ be defined by\
, \ \,for any \ $(r,s)\,\in\,J^{\,2}$\ ,
$$
f_{02}(r,s)\ =\ \big(\,r\,,\,f_{01}(s)\,\big)\ .
$$

For any compact connected manifold $M$,\ \ denote by \,$\partial M$
\,the boundary\,, \,and by $\stackrel{\ \circ}{M}$ \,the interior of
\,$M$. \ Specially, \ we have \ $\partial J\,
=\,\{\,-1\,,\,1\:\!\}$\,, \ and \ $\stackrel{\
\circ}{J}\,=\,(\,-1\,,1\:\!)$\,. \ Note that \ $\partial J^{\,2}$ \
is \ $\partial(J^{\,2})$\,,\ \ not \ $(\partial J)^{\,2}$\,.

\begin{defin}\label{normal rising} \ \ For any\ \ $s\,\in\,J$\,,\ \ \,write \ $J_s\,=
\,J\times\{s\}$\;. \ A homeomorphism \ $f:J^{\,2}\to\,J^{\,2}$ \ is
said to be \,{\it normally rising} \ if
$$
f\,|\;\partial J^{\,2}\,=\,f_{02}\,|\;\partial J^{\,2}\,,
\hspace{10mm}    \mbox{and} \hspace{8mm} f(J_s)\;=\,f_{02}(J_s)
\hspace{5mm} \mbox{for \ any}\ \ \,s\,\in\,J.
$$
\end{defin}

For $s\in J$, let $J_s=J\times \{s\}$ and $\mathcal C_s$ be the
collection of all nonempty connected closed subsets of $J_s$. A map
$\phi:J\rightarrow \mathcal C_s$ is {\it increasing} if the
abscissae and ordinates of the endpoints of $\phi(r)$ are increasing
functions of $r$ and is {\it endpoint preserving} if $\phi(-1)=(-1,
s)$ and $\phi(1)=(1, s)$. If $f$ is a normally rising homeomorphism
on $J^2$ and $s\in J$, then $\omega_{sf}(r)\equiv\omega_f(r, s)$
defines an increasing function $\omega_{sf}$ from $J$ to $\mathcal
C_1$. Similarly, we also have an increasing function $\alpha_{sf}$
from $J$ to $\mathcal C_{-1}$.
\medskip

Write $\mathcal A=\mathcal C_1$ and $\mathcal A'=\mathcal C_{-1}$. In
\cite{MSYZ24}, we obtained the following result.

\begin{thm}\label{thm:3-2}
Let \ \,$\mathbb N\,'$ and \ \,$\mathbb N\,''$\ \,be \,two
\,nonempty\, subsets\,\ of \ \,$\mathbb N\,$,\ \,and \;let
\;\,${\mathcal V}\,=\,\{\,V_n\,:\,n\, \in\,\mathbb N\,'\,\}$\ \ and
\ \ ${\mathcal W}\,=\,\{\,W_j\,:\,j\,\in\,\mathbb N\,''\,\}$ \ be
\,two\, families \,of \;pairwise\ \,disjoint \,nonempty\
\,connected\ \,subsets\ \,of \;the\, semi-open interval \ $(0,
1/2]$\;.\ \ For \,each \ $n\in\mathbb N\,'$\ \ and \,each \
$j\in\mathbb N\,''$\,,\ \ let \
\,$\omega_n\,:\,J\,\to\;\mathcal{A}$\ \; and \
\,${\alpha}_j\,:\,J\, \to \;\mathcal{A}'$ \ be \,given
\,increasing \,and\,\ endpoint \,preserving \,maps. \,Then \,there
\,exists \,a \,normally \,rising \,homeomorphism \
$f\,:\,J^{\,2}\,\to\;J^{\,2}$\ \,such\ \,that \
${\omega}_{sf\,}=\ {\omega}_n$\ \ for
\,any \ $n \in\mathbb N\,'$\ \ and \,any \ $s\in V_n$\ , \ and \
\,${\alpha}_{tf\,}=\ {\alpha}_j$ \ for
\,any \ $j \in\mathbb N\,''$\ \ and \,any \ $t\in W_j$.
\end{thm}

Specially, we have

\begin{cor}\label{thm:3-3}
Let \ \,$\mathbb N^{*}$ \,be \, a \,nonempty\, subset\,\ of \
\,$\mathbb N\,$. Then for any given maps $\omega: \mathbb{N}^{*}
\rightarrow \mathcal{A}$ and $\alpha: \mathbb{N}^{*} \rightarrow
\mathcal{A}'$, there exist a family $\{S_n: n \in \mathbb{N}^*\}$
of pairwise disjoint countable subsets of $\stackrel{\ \circ}{J}$
and a normally rising homeomorphism $f: J^2 \rightarrow J^2$ such
that, for any $n \in \mathbb{N}^{*}$, $S_n$ is dense in $J$, and
$\omega_f(r,s)=\omega(n)$ and $\alpha_f(r,s)=\alpha(n)$ for any
$(r,s) \in \stackrel{\ \circ}{J} \times S_n$.
\end{cor}

\begin{proof}
We may consider only the case that $\mathbb{N}^{*}=\mathbb{N}$. Take
a family $\{S_{n1}: n \in \mathbb{N}\}$ of pairwise disjoint
countable subsets of the semi-open interval $I_1=(0, 1/2]$ such that
every $S_{n1}$ is dense in $I_1$. Let $\mathbb{N}'=\mathbb{N}''=\mathbb{N}$,
$\mathcal{V}=\mathcal{W}=\{S_{n1}: n \in \mathbb{N}\}$, and for any $n \in
\mathbb{N}$ and any $s \in S_{n1}$, let the maps $\omega_s: J
\rightarrow \mathcal{A}$ and $\alpha_s: J \rightarrow \mathcal{A}'$
be defined by $\omega_s(r)=\omega(n)$ and $\alpha_s(r)=\alpha(n)$
for any $r \in \stackrel{\ \circ}{J}$. Then from Theorem
\ref{thm:3-2},  there exists a normally rising
homeomorphism $f: J^2 \rightarrow J^2$ such that
$$\omega_f(r,s)=\omega_s(r)=\omega(n) \mbox{\hspace{5mm} and \hspace{5mm}}
\alpha_f(r,s)=\alpha_s(r)=\alpha(n)$$ for any $n \in \mathbb{N}$ and
any $(r,s) \in \stackrel{\ \circ}{J} \times S_{n1}$. Let
$$S_n=\bigcup\{f_{01}^k(S_{n1}): k \in \mathbb{Z}\}.$$
Then $\{S_n: n \in \mathbb{N}\}$ is a family of pairwise disjoint
countable subsets of $\stackrel{\ \circ}{J}$, each $S_n$ is dense in
$J$, and $\omega_f(r,s)=\omega(n)$ and $\alpha_f(r,s)=\alpha(n)$ for
any $n \in \mathbb{N}$ and any $(r,s) \in \stackrel{\ \circ}{J}
\times S_n$. Corollary \ref{thm:3-3} is proved.
\end{proof}

Let $(X, d)$ and $(Y, \rho)$ be two metric spaces, and let $\lambda \geq
1$ be given. A bijection $h: X \rightarrow Y$ is called a {\it
$\lambda$-homeomorphism} if $d(x,y)/\lambda \leq \rho(h(x), h(y))
\leq \lambda d(x,y)$ for any $x, y \in X$.

\begin{lem} \label{lem:3-4}
Let $\{V_n: n \in \mathbb{N}\}$ be a family of pairwise disjoint
countable subset of $\stackrel{\ \circ}{J}\!^2$, and $\{W_n: n \in
\mathbb{N}\}$ be a family of pairwise disjoint subsets of
$\stackrel{\ \circ}{J}\!^2$. Suppose that, for each $n \in
\mathbb{N}$, both $V_n$ and $W_n$ are dense in $J^2$. Then for any
$\lambda>1$ and any $\varepsilon>0$, there exists a
$\lambda$-homeomorphism $h: J^2 \rightarrow J^2$ such that
$h|\partial J^2=id$, $d(h(x), x)<\varepsilon$ for any $x \in J^2$,
and $h(V_n) \subset W_n$ for any $n \in \mathbb{N}$.

\begin{proof}
Let $\mathbf{V}=\bigcup\{V_n: n \in \mathbb{N}\}$. Then $\mathbf{V}$
is also countable. Take a bijection $\chi: \mathbb{N} \rightarrow
\mathbf{V}$ and let $x_k=\chi(k)$ for any $k \in \mathbb{N}$. Define
a map $\gamma: \mathbf{V} \rightarrow \mathbb{N}$ by $\gamma(x)=n$
for any $x \in V_n$ and for any $n \in \mathbb{N}$. Take
$\lambda_1>\lambda_2>\lambda_3>\cdots>1$ and
$\varepsilon_1>\varepsilon_2>\varepsilon_3>\cdots>0$ such that
$\Pi_{j=1}^\infty \lambda_j<\lambda$ and
$\sum_{j=1}^{\infty}\varepsilon_j<\varepsilon$. For any $k \in
\mathbb{N}$, write
$$\mu_k=\prod_{j=1}^k \lambda_j \mbox{\hspace{8mm} and \hspace{8mm}}
\delta_k=\sum_{j=1}^k \varepsilon_j.$$

{\bf Claim 1.}\ For any $k \in \mathbb{N}$, there exists a
$\mu_k$-homeomorphism $h_k: J^2 \rightarrow J^2$ such that

\medskip

(C.k.1)\ $h_k|\partial J^2=id$;

\medskip

(C.k.2)\ $d(h_k(x),x)<\delta_k$ for any $x \in J^2$;

\medskip

(C.k.3)\ $h_k(x_k) \in W_{\gamma(x_k)}$ and $h_k(x_j)=h_j(x_j)$ for
$k>j \geq 1$.

\medskip

{\bf Proof of Claim 1.}\; Let $\tau_1=d(x_1, \partial J^2)$. Take a
point $y_1 \in W_{\gamma(x_1)}$ such that
$$d(y_1, x_1)<\min\{\varepsilon_1, (\lambda_1-1)\tau_1/\lambda_1\}.$$
Clearly, there exists a $\mu_1$ ($=\lambda_1$)-homeomorphism $h_1:
J^2 \rightarrow J^2$ such that $h_1(x_1)=y_1$,
$h_1|\big(J^2-\mathrm{Int}\;B(x_1, \tau_1)\big)=id$, and
$d(h_1(x), x)<\varepsilon_1=\delta_1$ for any $x \in J^2$.

\medskip

We now assume that, for some $m>1$,  there
exist $\mu_k$-homeomorphisms $h_k: J^2 \rightarrow J^2$ ($k=1, 2,
\cdots, m-1$) such that, for $m>k\geq
1$, the conditions (C.k.1)--(C.k.3) hold. Let $y_k=h_{m-1}(x_k)=h_k(x_k)$ for $k \in
\mathbb{N}_{m-1}$, let $z_m=h_{m-1}(x_m)$, and let
$$\tau_m=d\big(z_m, \partial J^2 \cup \{y_k: k \in \mathbb{N}_{m-1}\}\big).$$
Then $\tau_m>0$. Take a point $y_m \in W_{\gamma(x_m)}$ such that
$$d(y_m, z_m)<\min\{\varepsilon_m, (\lambda_m-1)\tau_m/\lambda_m\}.$$
Clearly, there exists a $\lambda_m$-homeomorphism $g_m: J^2
\rightarrow J^2$ such that $g_m(z_m)=y_m$,
$g_m|\big(J^2-\mathrm{Int}\; B^2(z_m, \tau_m)\big)=id$, and
$d(g_m(x),x)<\varepsilon_m$ for any $x \in J^2$. Take
$$h_m=g_mh_{m-1}: J^2 \rightarrow J^2.$$
Then $h_m$ is a $\mu_m$\ ($=\lambda_m\mu_{m-1}$)-homeomorphism, and
the conditions (C.k.1)--(C.k.3) for $k=m$ also hold. By induction,
Claim 1 is proved.

\medskip

The sequence of homeomorphisms $h_1, h_2, h_3, \cdots$ is uniformly
convergent. Let $$h=\lim_{k \rightarrow \infty}h_k: J^2 \rightarrow
J^2.$$Then $h$ satisfies the conditions in Lemma \ref{lem:3-4}, and
the proof is completed.
\end{proof}
\end{lem}

From Corollary \ref{thm:3-3} and Lemma \ref{lem:3-4}, we get

\begin{thm} \label{thm:3-5}
Let $\{V_n: n \in \mathbb{N}\}$ be a given family of pairwise
disjoint countable subsets of $\stackrel{\ \circ}{J}$$^{\;2}$, in
which each $V_n$ is dense in $J^2$. Then for any given maps $\omega:
\mathbb{N} \rightarrow \mathcal{A}$ and $\alpha: \mathbb{N}
\rightarrow \mathcal{A}'$, there exists a homeomorphism $\varphi:
J^2 \rightarrow J^2$ such that, for any $n \in \mathbb{N}$ and any
$x \in V_n$, one has $\omega_{\varphi}(x)=\omega(n)$ and
$\alpha_{\varphi}(x)=\alpha(n)$.
\end{thm}

\begin{proof}
Let $\{S_n: n \in \mathbb{N}\}$ and $f: J^2 \rightarrow J^2$ be the
same as in Corollary \ref{thm:3-3}. For any $n \in \mathbb{N}$,
write $W_n=\stackrel{\ \circ}{J} \times S_n$. Then $W_n$ is an
uncountable subsets of $\stackrel{\ \circ}{J}$$^{\;2}$, which is
dense in $J^2$. Let $h: J^2 \rightarrow J^2$ be the same as in Lemma
\ref{lem:3-4}. Take $\varphi=h^{-1}fh: J^2 \rightarrow J^2$. Then
$\varphi$ satisfies the conditions in Theorem \ref{thm:3-5}.
\end{proof}

\begin{rem}\label{main-sensitive}
Clearly, if the families $\mathcal A$ and $\mathcal A'$ consist of infinitely many
pairwise disjoint closed sets, then $f$ appearing in Corollary \ref{thm:3-3} and
$\varphi$ appearing in Theorem \ref{thm:3-5} are $(\omega, \alpha, \infty)$-sensitive.
Since the restrictions of $f$ and $\varphi$ to their nonwandering sets are identities,
the topological entropies $h_{top}(f)=h_{top}(\varphi)=0$ (see \cite[Corollary 8.6.1]{Wa82}).
Thus Theorem \ref{main1} is proved.
\end{rem}

\section{Permeating from the boundary of $J^2$ to its interior}\label{Sec-4}

\begin{defin} \label{handing set}
Let $\mathbf{X}$ be the family of all nonempty compact connected
sets in $J^2$, and let $s \in \{-1,1\}$.

\medskip

(1)\ \ A set $X \in \mathbf{X}$ is called a {\it \textbf{hanging
set}} on $J_s$ if $X \cap \partial J^2$ is a nonempty connected set
contained in $\stackrel{\ \circ}{J_s}$. Write
$$\mathbf{X}(s)=\{X \in \mathbf{X}: X \mbox{\ is a hanging set on } J_s\},$$
$$\mathbf{X}(s,0)=\{X \in \mathbf{X}(s): X \subset \stackrel{\ \circ}{J_s}\},$$
$$\mathbf{X}(s, Pt)=\{X \in \mathbf{X}(s): X \cap \partial J^2 \mbox{\ has only one point}\},$$
$$\mathbf{X}(s, Ac)=\{X \in \mathbf{X}(s): X \cap \partial J^2 \mbox{\ is an arc}\}.$$

(2)\ \ A set $Y \in \mathbf{X}$ is said to be {\it
\textbf{floating}} if $Y \cap \partial J^2=\emptyset$. Write
$$\mathbf{Y}=\{Y \in \mathbf{X}: Y \mbox{\ is a floating set}\}.$$
For any $Y \in \mathbf{Y}$, an arc $L \subset J^2$ is called a {\it
\textbf{cable from $Y$ to $J_s$}} if $L$ has an end point lying on
$\stackrel{\ \circ}{J_s}$, the other end point of $L$ lies on $Y$,
and $\stackrel{\ \circ}{L} \cap (\partial J^2 \cup Y)=\emptyset$.
Write
$$\mathbf{L}(Y, s)=\{L: L \mbox{\ is a cable from\ } Y \mbox{\ to\ } J_s\},$$
$$\mathbf{X}(Y, s)=\{Y \cup L: L \in \mathbf{L}(Y, s)\},$$
$$\mathbf{X}(\mathbf{Y}, s)=\bigcup \{\mathbf{X}(Y, s): Y \in \mathbf{Y}\}.$$
Then $\mathbf{X}(\mathbf{Y}, s) \subset \mathbf{X}(s,
Pt)-\mathbf{X}(s,0)$.

\medskip

(3)\ \ For any $X \in \mathbf{X}(s)$, a disc $D \in \mathbf{X}(s,
Ac)$ is called a {\it \textbf{hutch}} of $X$ if $X \cap J_s \subset
\mathrm{Int}(D \cap J_s)$ and $X-J_s \subset \stackrel{\ \circ}{D}$.
\end{defin}

\begin{defin} \label{permeable}
Let $X \in \mathbf{X}(s)$ where $s=-1$ or $1$, and let $D \in \mathbf{X}(s,
Ac)$ be a hutch of $X$. Write $A=D \cap J_s$. If there exists a
continuous map $\eta: D \rightarrow D$ such that

\medskip

(C.1)\ \ $\eta|(\partial D-\stackrel{\ \circ}{A})=id$, $\eta(A)=A
\cup X$, and there is an arc $A' \subset \stackrel{\ \circ}{A}$ such
that $X \cap J_s \subset \mathrm{Int}~A'$, $\eta(A-A')=A-(X \cap
J_s)$, $\eta(A')=X$, and $\eta|(A-A')$ is injective;

\medskip

(C.2)\ \ $\eta(\stackrel{\ \circ}{D})=\stackrel{\ \circ}{D}-X$, and
$\eta|\stackrel{\ \circ}{D}$ is injective,
\vspace{3mm}\\
then $\eta$ is called a {\it \textbf{permeating}} from $A$ to $X$,
and $X$ is called a {\it \textbf{permeable hanging set}}.
\end{defin}

As an example of permeable hanging sets, we have the following
result.

\begin{lem} \label{lem:4-3}
Let $X$ be a hanging set on $J_s$ where $s=-1$ or $1$, and let $D$ be a hutch
of $X$. If $X=T$ is a tree, then there exists a continuous map
$\eta: D \rightarrow D$ satisfying the conditions (C.1) and (C.2) in
Definition \ref{permeable}.
\end{lem}
\begin{proof}[Outline of proof]
The proof of Lemma \ref{lem:4-3} is not hard, and we give only
the following Figure 4.1 to illustrate the idea of the proof.
Notice that the map $\eta_1$ is a homeomorphism from $D$ to $D'$
and $\eta_2$ is a continuous map defined by identifying appropriately the points
in the boundary of $U$. Then $\eta=\eta_2\eta_1:D\rightarrow D$ is a permeating.
\end{proof}
\vspace{10mm}

\begin{center}
\setlength{\unitlength}{0.6mm}
\begin{picture}(220,60)

\put(0,60){\line(1,0){10}}
\multiput(10,60)(0.1,0){100}{\circle*0.2}
\multiput(20,60)(1,0){15}{\circle*0.2}
\multiput(35,60)(0.1,0){150}{\circle*0.2}
\put(50,60){\line(1,0){10}}
\put(0,30){\line(0,1){30}}
\put(60,30){\line(0,1){30}}
\put(0,30){\line(1,-1){15}}
\put(60,30){\line(-1,-1){15}}
\put(15,15){\line(1,0){30}}
\multiput(25,60)(-0.5,-1){10}{\circle*0.2}
\multiput(30,60)(0,-1){15}{\circle*0.2}
\multiput(30,45)(-1,-1){10}{\circle*0.2}
\multiput(30,45)(1,-1){6}{\circle*0.2}
\multiput(25,40)(1,-1){4}{\circle*0.2}
\put(28,20){\scriptsize$D$}
\put(33,47){\scriptsize$T$}

\put(80,60){\line(1,0){20}}
\multiput(100,60)(0.1,0){40}{\circle*0.2}
\multiput(106,60)(0.1,0){30}{\circle*0.2}
\multiput(111,60)(0.1,0){40}{\circle*0.2}
\put(115,60){\line(1,0){25}}
\put(80,30){\line(0,1){30}}
\put(140,30){\line(0,1){30}}
\put(80,30){\line(1,-1){15}}
\put(140,30){\line(-1,-1){15}}
\put(95,15){\line(1,0){30}}
\multiput(104,60)(-0.1,-0.2){40}{\circle*0.2}
\multiput(106,60)(-0.1,-0.2){40}{\circle*0.2}
\multiput(100,52)(0.1,0){20}{\circle*0.2}
\multiput(109,60)(0,-0.1){150}{\circle*0.2}
\multiput(111,60)(0,-0.1){150}{\circle*0.2}
\multiput(109,45)(-0.1,-0.1){90}{\circle*0.2}
\multiput(111,45)(0.1,-0.1){50}{\circle*0.2}
\multiput(110,43)(0.1,-0.1){60}{\circle*0.2}
\multiput(116,37)(0,0.1){30}{\circle*0.2}
\multiput(110,43)(-0.1,-0.1){40}{\circle*0.2}
\multiput(104,37)(-0.1,-0.1){40}{\circle*0.2}
\multiput(100,33)(0,0.1){30}{\circle*0.2}
\multiput(104,37)(0.1,-0.1){40}{\circle*0.2}
\multiput(106,39)(0.1,-0.1){40}{\circle*0.2}
\multiput(110,35)(-0.1,-0.1){20}{\circle*0.2}
\put(108,20){\scriptsize$D'$}

\put(160,60){\line(1,0){20}}
\multiput(180,60)(0.1,0){150}{\circle*1}
\put(195,60){\line(1,0){25}}
\put(160,60){\line(0,-1){30}}
\put(220,60){\line(0,-1){30}}
\put(160,30){\line(1,-1){15}}
\put(175,15){\line(1,0){30}}
\put(205,15){\line(1,1){15}}
\multiput(185,60)(-0.05,-0.1){100}{\circle*1}
\multiput(190,60)(0,-0.1){150}{\circle*1}
\multiput(190,45)(-0.1,-0.1){100}{\circle*1}
\multiput(190,45)(0.1,-0.1){60}{\circle*1}
\multiput(185,40)(0.1,-0.1){40}{\circle*1}
\put(188,20){\scriptsize$D$}
\put(193,47){\scriptsize$T$}

\put(60,17){\vector(1,0){20}} \put(140,17){\vector(1,0){20}}
\put(69,22){\scriptsize$\eta_1$} \put(149,22){\scriptsize$\eta_2$}
\put(2,5){\scriptsize{$T$ is a tree hanging on $J_1$,}}
\put(2,-1){\scriptsize{and $D$ is a hutch of $T$.}}
\put(81,5){\scriptsize{Dig out a neighborhood $U$}}
\put(81,-1){\scriptsize{of $T$, and obtain a disc $D'$.}}
\put(159,5){\scriptsize{Identify appropriately points }}
\put(159,-1){\scriptsize{of $\partial U$, and the tree $T$
reappears.}} \put(96,-15){{Figure 4.1}}
\end{picture}
\end{center}

\vspace{10mm}

\begin{lem} \label{lem:4-4}
Let $X$ be a permeable hanging set on $J_s$ where $s=-1$ or $1$, and let $D$
be a hutch of $X$. Let $A=D \cap J_s$, $A' \subset \stackrel{\
\circ}{A}$ and $\eta: D \rightarrow D$ be as in Definition
\ref{permeable}. If $X \in \mathbf{X}(\mathbf{Y}, s)$, that is,
there is a floating set $Y \in \mathbf{Y}$ and a cable $L$ from $Y$
to $J_s$ such that $X=Y \cup L$, then there exists an arc $A''
\subset \mathrm{Int}~A'$ such that $\eta(A'')=Y$.
\end{lem}

The proof of Lemma \ref{lem:4-4} is not hard, and is omitted.

\begin{prop} \label{prop:4-5}
Let $X$ be a hanging set on $J_s$, where $s=-1$ or $1$. If $X$ is
permeable, then $X$ is a dendrite.
\end{prop}

\begin{proof}
Let $D$, $A$, $A'$ and $\eta: D \rightarrow D$ be the same as in
Definition \ref{permeable}. Since $X=\eta(A')$ is the continuous
image of arc $A'$, $X$ is a Peano continuum by  \cite[Theorem 8.18]{Na92}. For any point $x \in X$, take a point $y \in A'$ such that
$\eta(y)=x$. Then for any neighborhood $U$ of $x$ in $D$, there is a
neighborhood $V$ of $y$ in $D$ such that $\eta(V) \subset U$. Since
$V \cap \stackrel{\ \circ}{D} ~\neq \emptyset$ and $\eta(V \cap
\stackrel{\ \circ}{D}) \subset \eta(\stackrel{\
\circ}{D})~=~\stackrel{\ \circ}{D}-X$, we have
$$U-X \supset \eta(V)-X \supset \eta\big(V \cap \stackrel{\ \circ}{D}\big)
-X=\eta\big(V \cap \stackrel{\ \circ}{D}\big) ~\neq~ \emptyset.$$
Thus any neighborhood $U$ of any $x \in X$ is not contained in $X$,
and hence we get

\medskip

{\bf Claim 1.}\ \ $X$ contains no disc.

\medskip

If $X$ contains a circle $C$, then
\begin{equation} \label{eq:4-1}
C \cap \eta\big(\stackrel{\ \circ}{D}\big)= C \cap \big(\stackrel{\
\circ}{D}-X\big) \subset X - \big(\stackrel{\ \circ}{D}-X\big) =
\emptyset.
 \end{equation}
By Claim 1, we get
\begin{equation} \label{eq:4-2}
\mathrm{Int}~\mathrm{Dsc}(C) \cap \eta\big(\stackrel{\
\circ}{D}\big) = \mathrm{Int}~\mathrm{Dsc}(C) \cap \big(\stackrel{\
\circ}{D}-X\big) \neq \emptyset.
 \end{equation}
For any arc $A^* \subset \partial D-A$, we have
$$d(A^*, \mathrm{Dsc}(C))=d(A^*, C) \geq d(A^*, X)>0,$$
which with $\eta|(\partial D-A)=id$ implies
\begin{equation} \label{eq:4-3}
\eta\big(\stackrel{\ \circ}{D}\big)-\mathrm{Dsc}(C) \neq \emptyset.
 \end{equation}
From \eqref{eq:4-1}--\eqref{eq:4-3} we see that
$\eta\big(\stackrel{\ \circ}{D}\big)$ is not connected. On the other
hand, $\stackrel{\ \circ}{D}$ is connected, and hence
$\eta\big(\stackrel{\ \circ}{D}\big)$ is also connected. This leads
to a contradiction. Thus $X$ cannot contain a circle, and hence $X$
is a dendrite. Proposition \ref{prop:4-5} is proved.
\end{proof}

As an inverse of Proposition \ref{prop:4-5}, we leave the following
question.

\begin{ques} \label{conj:4-6}
Let $X$ be a dendrite hanging on $J_s$, where $s=-1$ or $1$. Must $X$ be
permeable?
\end{ques}

In Theorem \ref{thm:3-5}, the $\omega$-limit sets and the
$\alpha$-limit sets of points $x \in \stackrel{\
\circ}{J}\mbox{$^2$}$ under the homeomorphism $\varphi: J^2
\rightarrow J^2$ are always in $J_1$ and $J_{-1}$, respectively. By
means of the permeating $\xi: J^2 \rightarrow J^2$ we can obtain a
homeomorphism $\psi: J^2 \rightarrow J^2$ such that the
$\omega$-limit sets and the $\alpha$-limit sets of   $x \in
\stackrel{\ \circ}{J}\mbox{$^2$}$ under $\psi$ need not be confined
in $J_1$ or $J_{-1}$.

\begin{thm} \label{thm:4-7}
For $n \in \mathbb{N}$, let $X_{n1} \in \mathbf{X}(1)$ be a
permeable hanging set on $J_1$, $X_{n2} \in \mathbf{X}(-1)$ be a
permeable hanging set on $J_{-1}$, and $D_{n1}$ be a hutch of
$X_{n1}$, $D_{n2}$ be a hutch of $X_{n2}$. Suppose that
$$\lim_{n \rightarrow \infty} \mathrm{diam}(D_{n1})=
\lim_{n \rightarrow \infty} \mathrm{diam}(D_{n2})=0,$$ and the discs
in the family $\{D_{n1}, D_{n2}: n \in \mathbb{N}\}$ are pairwise
disjoint. Write
$$X_{\mathbb{N}}=\bigcup\{X_{n1} \cup X_{n2}: n \in \mathbb{N}\}.$$

Let $\mathbb{N}'$ and $\mathbb{N}''$ be two subsets of $\mathbb{N}$
such that $\{X_{n1}: n \in \mathbb{N}'\} \subset
\mathbf{X}(\mathbf{Y}, 1)$ and $\{X_{n2}: n \in \mathbb{N}''\}
\subset \mathbf{X}(\mathbf{Y}, -1)$, that is, for each $n \in
\mathbb{N}'$ there is a floating set $Y_{n1} \in \mathbf{Y}$ and a
cable $L_{n1}$ from $Y_{n1}$ to $J_1$ such that $X_{n1}=Y_{n1} \cup
L_{n1}$; for each $n \in \mathbb{N}''$ there is a floating set
$Y_{n2} \in \mathbf{Y}$ and a cable $L_{n2}$ from $Y_{n2}$ to
$J_{-1}$ such that $X_{n2}=Y_{n2} \cup L_{n2}$.

\medskip

Let $\{W_n: n \in \mathbb{N}\}$ be a family of pairwise disjoint
countable set in $\stackrel{\ \circ}{J}$$^{\;2}-X_{\mathbb{N}}$, in
which each $W_n$ is dense in $J^2$. Then there exists a
homeomorphism $\psi: J^2 \rightarrow J^2$ such that

\medskip

{\rm (1)}\ \ For any $n \in \mathbb{N}-\mathbb{N}'$ and any $w \in
W_n$, one has $\omega_{\psi}(w)=X_{n1}$;

\medskip

{\rm (2)}\ \ For any $n \in \mathbb{N}'$ and any $w \in W_n$, one
has $\omega_{\psi}(w)=Y_{n1}$;

\medskip

{\rm (3)}\ \ For any $n \in \mathbb{N}-\mathbb{N}''$ and any $w \in
W_n$, one has $\alpha_{\psi}(w)=X_{n2}$;

\medskip

{\rm (4)}\ \  For any $n \in \mathbb{N}''$ and any $w \in W_n$, one
has $\alpha_{\psi}(w)=Y_{n2}$.
\end{thm}

\begin{proof}
For $n \in \mathbb{N}$, write $A_{n1}=D_{n1} \cap J_1$,
$A_{n2}=D_{n2} \cap J_{-1}$, and for $i \in \mathbb{N}_2$, let
$\eta_{ni}: D_{ni} \rightarrow D_{ni}$ be a permeating from $A_{ni}$
to $X_{ni}$. Define a map $\xi: J^2 \rightarrow J^2$ by
$$\xi|D_{n1}=\eta_{n1} \mbox{\ \ and\ \ } \xi|D_{n2}=\eta_{n2} \mbox{\ \ for any\ } n \in
\mathbb{N},$$ and
$$\xi|\big(J^2-\bigcup\{D_{n1} \cup D_{n2}: n \in \mathbb{N}\}\big)=id.$$
Then $\xi$ is continuous, called a {\it \textbf{compositive
permeating}} from $J_1 \cup J_{-1}$ to $X_{\mathbb{N}}$.

\medskip

Note that $W_n \subset \stackrel{\
\circ}{J}\mbox{$^2$}-X_{\mathbb{N}} \subset J^2-(J_1 \cup J_{-1}
\cup X_{\mathbb{N}})$ for all $n \in \mathbb{N}$. Let
$V_n=\xi^{-1}(W_n)$. Then $\{V_n: n \in \mathbb{N}\}$ is a family of
pairwise disjoint countable sets in $\stackrel{\
\circ}{J}\mbox{$^2$}=\xi^{-1}\big(\stackrel{\
\circ}{J}\mbox{$^2$}-X_{\mathbb{N}}\big)$, and each $V_n$ is dense
in $J^2$. Using Lemma \ref{lem:4-4}, we take maps $\omega:
\mathbb{N} \rightarrow \mathcal{A}$ and $\alpha: \mathbb{N}
\rightarrow \mathcal{A}'$ as follows:

\medskip

(a)\ \ For any $n \in \mathbb{N}-\mathbb{N}'$, let $A'_{n1}$ be an
arc in $\stackrel{\ \circ}{A_{n1}}$ such that $\xi(A'_{n1})=X_{n1}$
and we take $\omega(n)=A'_{n1}$;

\medskip

(b)\ \ For any $n \in \mathbb{N}'$, let $A''_{n1}$ be an arc in
$\stackrel{\ \circ}{A_{n1}}$ such that $\xi(A''_{n1})=Y_{n1}$ and we
take $\omega(n)=A''_{n1}$;

\medskip

(c)\ \ For any $n \in \mathbb{N}-\mathbb{N}''$, let $A'_{n2}$ be an
arc in $\stackrel{\ \circ}{A_{n2}}$ such that $\xi(A'_{n2})=X_{n2}$
and we take $\alpha(n)=A'_{n2}$;

\medskip

(d)\ \ For any $n \in \mathbb{N}''$, let $A''_{n2}$ be an arc in
$\stackrel{\ \circ}{A_{n2}}$ such that $\xi(A''_{n2})=Y_{n2}$ and we
take $\alpha(n)=A''_{n2}$.

\medskip

By Theorem \ref{thm:3-5}, there exists a normally rising
homeomorphism $\varphi: J^2 \rightarrow J^2$ such that, for any $n
\in \mathbb{N}$ and any $v \in V_n$, one has
$\omega_{\varphi}(v)=\omega(n)$ and $\alpha_{\varphi}(v)=\alpha(n)$.

\medskip

Let $\psi=\xi \varphi \xi^{-1}$.

\medskip

Strictly speaking, both $\xi^{-1}$ and $\xi \varphi \xi^{-1}$ are
maps from $J^2$ to the family of all subsets of $J^2$, not maps from
$J^2$ to $J^2$. However, for some $y \in Y \subset J^2$, if
$\xi^{-1}(y)$ contains just one point $x$, then we may write
$\xi^{-1}(y)=x$ for $\xi^{-1}(y)=\{x\}$. If for each $z \in Y$,
$\xi^{-1}(z)$ contains just one point, then we may regard
$\xi^{-1}|Y$ as a map from $Y$ to $J^2$ or a bijection from $Y$ to
$\xi^{-1}(Y)$. Under this definition, from Definition
\ref{permeable} we get

\medskip

{\bf Claim 1.}\ \ $\xi^{-1}|\big(J^2-(J_1 \cup J_{-1} \cup
X_{\mathbb{N}})\big)$ is a bijection from $J^2-(J_{1} \cup J_{-1}
\cup X_{\mathbb{N}})$ to $J^2-J_1-J_{-1}$.

\medskip

Handing $\psi=\xi \varphi \xi^{-1}$ similarly, we have

\medskip

{\bf Claim 2.}\ \ $\psi$ is a homeomorphism from $J^2$ to $J^2$.

\medskip

{\bf Proof of Claim 2.}\ \ Consider any $x \in J^2$. If $x \in
J^2-(J_1 \cup J_{-1} \cup X_{\mathbb{N}})$, then $\xi^{-1}(x)$
contains only one point, and hence $\xi\varphi\xi^{-1}(x)$ also
contains only one point. If $x \in J_1 \cup J_{-1} \cup
X_{\mathbb{N}}$, then $\xi^{-1}(x)$ may contain more that one point,
but in this case $\xi^{-1}(x) \subset J_1 \cup J_{-1} \cup
\mathrm{Fix}(\varphi)$, which implies that $\xi \varphi \xi^{-1}(x)$
still contains only one point, and this point is $x$ itself. Thus
$\psi=\xi \varphi \xi^{-1}$ is a map from $J^2$ to $J^2$, and we
have
\begin{equation} \label{eq:4-4}
\psi|(J_1 \cup J_{-1} \cup X_{\mathbb{N}})=id.
 \end{equation}

Since $\varphi|(J^2-J_1-J_{-1})$ is also a bijection from
$J^2-J_1-J_{-1}$ to itself, it follows from Claim 1 that
$\psi|\big(J^2-(J_1 \cup J_{-1} \cup X_{\mathbb{N}})\big)$ is a
bijection from $J^2-(J_1 \cup J_{-1} \cup X_{\mathbb{N}})$ to
itself, which with \eqref{eq:4-4} implies that $\psi=\xi \varphi
\xi^{-1}: J^2 \rightarrow J^2$ is a bijection.

\medskip

For any closed set $Z$ in $J^2$, $\xi^{-1}(Z)$,
$\varphi^{-1}\xi^{-1}(Z)$ and
$\psi^{-1}(Z)=\xi\varphi^{-1}\xi^{-1}(Z)$ are all closed sets in
$J^2$. Thus $\psi$ is continuous, and hence $\psi$ is a
homeomorphism. Claim 2 is proved.

\medskip

Note that, for any $x \in J^2$ and $y=\xi(x)$, if $\xi^{-1}(y)$
contains more than one point, then $x=\xi^{-1}\xi(x)$ does not hold,
and we have only $x \in \xi^{-1}\xi(x)$. However, even if
$\xi^{-1}(y)$ contains more than one point, since
$\psi(y)=\xi\varphi\xi^{-1}(y)$ contains only one point, $\xi
\varphi\xi^{-1}\xi(x)=\xi\varphi(x)$ still holds. Hence we have
$\psi^2=\xi\varphi\xi^{-1}\xi\varphi\xi^{-1}=\xi\varphi^2\xi^{-1}$,
and in general, we have
\begin{equation} \label{eq:4-5}
\psi^n=\xi \varphi^n \xi^{-1} \mbox{\ \ for any\ } n \in \mathbb{N}.
 \end{equation}

Similarly, let $\psi^*=\xi\varphi^{-1}\xi^{-1}: J^2 \rightarrow
J^2$. Then $\psi^*$ is also a homeomorphism. For any $x \in J^2$,
since
$$\psi^*\psi(x)=\xi\varphi^{-1}\xi^{-1}\xi\varphi\xi^{-1}(x)=\xi\xi^{-1}(x)=x$$
and $\psi\psi^*\psi(x)=\psi(x)$, we have
$\psi^{-1}=\psi^*=\xi\varphi^{-1}\xi^{-1}$.

\medskip

{\bf Claim 3.}\ \ $\omega_{\psi}(\xi(x))=\xi(\omega_{\varphi}(x))$
for any $x \in J^2$.

\medskip

{\bf Proof of Claim 3.}\ \ For any $y \in \omega_{\varphi}(x)$,
there exist integers $0<n(1)<n(2)<n(3)<\cdots$ such that $\lim_{k
\rightarrow \infty} \varphi^{n(k)}(x)=y$, which with \eqref{eq:4-5}
implies
$$\lim_{k \rightarrow \infty}\psi^{n(k)}(\xi(x))=\lim_{k \rightarrow \infty}\xi\varphi^{n(k)}(x)=\xi(y).$$
This means that $\xi(y) \in \omega_{\psi}(\xi(x))$ and hence
$\xi(\omega_{\varphi}(x)) \subset \omega_{\psi}(\xi(x))$.

\medskip

Conversely, for any $y \in \omega_{\psi}(\xi(x))$, there exist
integers $0<n(1)<n(2)<n(3)<\cdots$ such that $\lim_{k \rightarrow
\infty} \psi^{n(k)}(\xi(x))=y$, which with \eqref{eq:4-5} implies
$$\lim_{k \rightarrow \infty} \xi \varphi^{n(k)}(x)=\lim_{k \rightarrow \infty}
\xi\varphi^{n(k)}\xi^{-1}(\xi(x))=\lim_{k \rightarrow
\infty}\psi^{n(k)}(\xi(x))=y.$$ Take integers
$0<k(1)<k(2)<k(3)<\cdots$ and $z \in \omega_{\varphi}(x)$ such that
$$\lim_{j \rightarrow \infty} \varphi^{n(k(j))}(x)=z,$$
we get $\xi(z)=y$. Thus $y \in \xi(\omega_{\varphi}(x))$ and hence
$\omega_{\psi}(\xi(x)) \subset \xi(\omega_{\varphi}(x))$. Claim 3 is
proved.

\medskip

Similarly, we have

\medskip

{\bf Claim 4.}\ \ $\alpha_{\psi}(\xi(x))=\xi(\alpha_{\varphi}(x))$
for any $x \in J^2$.

\medskip

We now verify that the homeomorphism $\psi: J^2 \rightarrow J^2$ has
the  properties (1)--(4):

\medskip

(1) For any $n \in \mathbb{N}-\mathbb{N}'$ and any $w \in W_n$, from
the definition of $\omega: \mathbb{N} \rightarrow \mathcal{A}$ with
case (a) we have $\omega_{\varphi}(\xi^{-1}(w))=\omega(n)=A'_{n1}$.
By Claim 3 we get
$$\omega_{\psi}(w)=\xi(\omega_{\varphi}(\xi^{-1}(w)))=\xi(A'_{n1})=X_{n1}.$$

(2) For any $n \in \mathbb{N}'$ and any $w \in W_n$, from the
definition of $\omega: \mathbb{N} \rightarrow \mathcal{A}$ with case
(b) we have $\omega_{\varphi}(\xi^{-1}(w))=\omega(n)=A''_{n1}$. By
Claim 3 we get
$$\omega_{\psi}(w)=\xi(\omega_{\varphi}(\xi^{-1}(w)))=\xi(A''_{n1})=Y_{n1}.$$

By Claim 4, the verifying of the properties (3) and (4) is similar,
and is omitted. Theorem \ref{thm:4-7} is proved.
\end{proof}

\begin{rem} \label{rem:4-8}
In Corollary \ref{thm:3-3}, the index set $\mathbb N^*$ can be an infinite
subset of $\mathbb{N}$, and can also be a finite subset of
$\mathbb{N}$. Similarly, in Theorem \ref{thm:3-5} and Theorem
\ref{thm:4-7}, the index set $\mathbb{N}$ can be replaced by a
finite subset of $\mathbb{N}$ or of $\mathbb{N} \cup \{0\}$. Of
course, in Theorem \ref{thm:4-7}, if the index set $\mathbb{N}$ is
replaced by a finite subset of $\mathbb{N} \cup \{0\}$, then the
condition
$$\lim_{n \rightarrow \infty}\mathrm{diam}(D_{n1})=
\lim_{n \rightarrow \infty}\mathrm{diam}(D_{n2})=0$$
will automatically vanish. In addition, in Theorem \ref{thm:3-5} and
Theorem \ref{thm:4-7}, the countable sets $V_n$ and $W_n$ need not
be dense in $J^2$ or even infinite sets.
\end{rem}

\section{Limit sets in the open square $\stackrel{\ \circ}{J}$$^{\;2}$}

Define the maps \ $p\;,\;q\;:\,\Bbb{R}^{\,2}\to\,\Bbb{R}$ \ by \
$p(x)\,=\,r$ \ and \ $q(x)\,=\,s$\ \ for any \
$x\,=\,(r,s)\,\in\,\Bbb{R}^{\,2}$\,,\ \;that is,\ \ we denote by \
$p(x)$\ \ and \ $q(x)$\ \ the abscissa and the ordinate of
\,$x$\,,\,\ respectively\,. \,

\begin{thm} \label{cor:4-9}
Let $Y_{\mathbb{N}}=\{y_{n1}, y_{n2}: n \in \mathbb{N}\}$ be a
countable infinite set in $\stackrel{\ \circ}{J}$$^{\;2}$, and
$\{W_n: n \in \mathbb{N}\}$ be a family of pairwise disjoint
countable sets in $\stackrel{\ \circ}{J}\mbox{$^2$}-Y_{\mathbb{N}}$,
in which each $W_n$ is dense in $J^2$. Suppose that
\begin{enumerate}
\item[{\rm (1)}] $\lim_{n \rightarrow \infty} d(y_{n1}, J_1)=\lim_{n \rightarrow
\infty} d(y_{n2}, J_{-1})=0$;

\medskip

\item[{\rm (2)}] For any $\beta \in p(Y_{\mathbb{N}})$, there exists
$\varepsilon_{\beta}>0$ such that $Y_{\mathbb{N}} \cap
p^{-1}([\beta-\varepsilon_{\beta}, \beta+\varepsilon_{\beta}])$ is a
finite set.
\end{enumerate}
Then there exists a zero entropy homeomorphism $\psi: J^2 \rightarrow J^2$ such
that, for any $n \in \mathbb{N}$ and any $w \in W_n$, one has
$\omega_{\psi}(w)=\{y_{n1}\}$ and $\alpha_{\psi}(w)=\{y_{n2}\}$.
\end{thm}

\begin{proof}
We regard the real plane $\mathbb{R}^2$ and the complex plane
$\mathbb{C}$ as the same, that is, any point $(r \cos \theta, r \sin
\theta) \in \mathbb{R}^2$ and the point $r
\mathrm{e}^{\mathrm{i}\theta} \in \mathbb{C}$ are regarded as the
same. Write $W_{\mathbb{N}}=\bigcup\{W_n: n \in \mathbb{N}\}$, and
$$Z=\{(w-y)/|w-y|: w \in W_{\mathbb{N}}, \mbox{\ and\ } y \in Y_{\mathbb{N}}\},$$
where $|z|$ denote the modulus of the complex number $z \in
\mathbb{C}$. Then $Z$ is a countably infinite set in the unit circle
$\mathbb{S}^1$.

\medskip

From the condition (2) we see that, for any $\beta \in
p(Y_{\mathbb{N}})$, there exists $\delta_{\beta}>0$ such that
$p(Y_{\mathbb{N}}) \cap [\beta-5\delta_{\beta},
\beta+5\delta_{\beta}]=\{\beta\}$, and $[\beta-5\delta_{\beta},
\beta+5\delta_{\beta}] \subset \stackrel{\ \circ}{J}$.

\medskip

For any $n \in \mathbb{N}$ and $i \in \mathbb{N}_2$, write
$Y_{ni}=\{y_{ni}\}$. Then $Y_{ni} \in \mathbf{Y}$. For any $\beta
\in p(Y_{\mathbb{N}})$, write $Y_{\beta}=Y_{\mathbb{N}} \cap
p^{-1}(\beta)$. Then $Y_{\beta}$ is a nonempty finite set. Take
$\mu_{\beta} \in (0, \delta_{\beta}]$ such that
$$(\mu_{\beta}+\mathrm{i})/|\mu_{\beta}+\mathrm{i}| \in
\mathbb{S}^1-Z.$$ For any $y \in Y_{\beta}$, if $y=y_{n1}$ for some
$n \in \mathbb{N}$, then there is a unique point $x_y=x_{n1} \in
J_1$ such that
$$(x_y-y)/|x_y-y|=(\mu_{\beta}+\mathrm{i})/|\mu_{\beta}+\mathrm{i}|,$$
and we take $[y, x_y]$ as the cable from $\{y\}=Y_{n1}$ to $J_1$. If
$y=y_{n2}$ for some $n \in \mathbb{N}$, then there is a unique point
$x_y=x_{n2} \in J_{-1}$ such that
$$(y-x_y)/|y-x_y|=(\mu_{\beta}+\mathrm{i})/|\mu_{\beta}+\mathrm{i}|,$$
and we take $[y, x_y]$ as the cable from $\{y\}=Y_{n2}$ to $J_{-1}$.
Noting that $\{[y, x_y]: y \in Y_{\beta}\}$ is a family of pairwise
disjoint straight line segments, which are all contained in
$$(\beta-2\mu_{\beta}, \beta+2\mu_{\beta}) \times J \subset
(\beta-2\delta_{\beta}, \beta+2\delta_{\beta}) \times J$$ and do not
intersect $Y_{\mathbb{N}}$, and of which the slopes are all
$1/\mu_{\beta}$, we can take a family $\mathbf{D}_{\beta}=\{D_y: y
\in Y_{\beta}\}$ of pairwise disjoint discs such that, for each $y
\in Y_{\beta}$, $D_y$ is a hutch of $[y, x_y]$ contained in
$(\beta-2\delta_{\beta}, \beta+2\delta_{\beta}) \times J$, and
$\mathrm{diam}(D_y)<3|x_y-y|/2$. Let
$\mathbf{D}=\bigcup\{\mathbf{D}_{\beta}: \beta \in
p(Y_{\mathbb{N}})\}$. For any $\beta \in p(Y_{\mathbb{N}})$, $\gamma
\in p(Y_{\mathbb{N}})-\{\beta\}$ and any $y \in Y_{\beta}$, $z \in
Y_{\gamma}$, we have
$$d(D_y, D_z)>d([\beta-2\delta_{\beta}, \beta+2\delta_{\beta}],
[\beta-2\delta_{\gamma},
\beta+2\delta_{\gamma}])>\max\{\delta_{\beta},
\delta_{\gamma}\}>0.$$ Thus $\mathbf{D}$ is a family of pairwise
disjoint discs.

\medskip

For any $y \in Y_{\mathbb{N}}$, if $y=y_{n1}$ for some $n \in
\mathbb{N}$ then we write $D_{n1}$ for $D_y$, and if $y=y_{n2}$ for
some $n \in \mathbb{N}$ then we write $D_{n2}$ for $D_y$. From the
condition (1) of the theorem, we get
$$\lim_{n \rightarrow \infty}\mathrm{diam}(D_{n1})=
\lim_{n \rightarrow \infty}\mathrm{diam}(D_{n2})=0.$$ By Theorem
\ref{thm:4-7}, there exists a homeomorphism $\psi: J^2 \rightarrow
J^2$ having the properties described in Theorem \ref{cor:4-9}.
Clearly, the topological entropy of $\psi$ is zero, since its restriction
to the nonwandering set is identity.
\end{proof}

In Theorem \ref{cor:4-9}, the homeomorphisms $\psi: J^2 \rightarrow
J^2$ satisfy $$\psi\big(\stackrel{\
\circ}{J}\mbox{$^{2}$}\big)=\stackrel{\ \circ}{J}\mbox{$^2$}.$$ Thus
we obtain the homeomorphisms $\psi|\stackrel{\
\circ}{J}\mbox{$^2$}:\; \stackrel{\ \circ}{J}\mbox{$^2$} \rightarrow
\stackrel{\ \circ}{J}\mbox{$^2$}$.\; Define a homeomorphism\; $H:\;
\stackrel{\ \circ}{J}\mbox{$^2$} \rightarrow \mathbb{R}^2$ by
$$H(r,s)=\big(\tan(r\pi/2), \tan(s\pi/2)\big) \mbox{\ \ for any\ }
(r,s) \in \stackrel{\ \circ}{J}\mbox{$^2$}.$$ Then, by means of $H$
and $\psi|\stackrel{\ \circ}{J}\mbox{$^2$}$, we get some
homeomorphisms of $\mathbb{R}^2$, which have properties similar to
$\psi: J^2 \rightarrow J^2$. For example, we have

\begin{cor} \label{thm:5-1}
Let $X_{\mathbb{N}}=\{x_{n1}, x_{n2}: n \in \mathbb{N}\}$ be a
countable infinite set in $\mathbb{R}^2$, and $\{V_n: n \in
\mathbb{N}\}$ be a family of pairwise disjoint countable sets in
$\mathbb{R}^2-X_{\mathbb{N}}$, in which each $V_n$ is dense in
$\mathbb{R}^2$. Suppose that
\begin{enumerate}
\item[{\rm (1)}] $\lim_{n \rightarrow \infty} q(x_{n1})=\lim_{n \rightarrow
\infty} q(x_{n2})=0$;

\medskip

\item[{\rm (2)}] For any $\beta \in p(X_{\mathbb{N}})$, there exists
$\delta_{\beta}>0$ such that $X_{\mathbb{N}} \cap
p^{-1}([\beta-\delta_{\beta}, \beta+\delta_{\beta}])$ is a finite
set.
\end{enumerate}
Then there exists a homeomorphism $F: \mathbb{R}^2 \rightarrow
\mathbb{R}^2$ such that, for any $n \in \mathbb{N}$ and any $v \in
V_n$, one has $\omega_F(v)=\{x_{n1}\}$ and $\alpha_F(v)=\{x_{n2}\}$.
\end{cor}

\begin{proof}
For $n \in \mathbb{N}$, let $Y_{\mathbb{N}}=H^{-1}(X_{N})$,
$y_{n1}=H^{-1}(x_{n1})$, $y_{n2}=H^{-1}(x_{n2})$, and
$W=H^{-1}(V_n)$. Then the conditions (1) and (2) in Theorem
\ref{cor:4-9} hold, and there exists a homeomorphism $\psi: J^2
\rightarrow J^2$ having the properties described in this theorem.
Define $F: \mathbb{R}^2 \rightarrow \mathbb{R}^2$ by $F=H\psi
H^{-1}$. Then $F$ meets the requirement.
\end{proof}

\begin{thm} \label{cor:4-10}
For any given $m \in \mathbb{N}$, let $\{T_{n1}, T_{n2}: n \in
\mathbb{N}_{m}\}$ be a family of pairwise disjoint trees in
$\stackrel{\ \circ}{J}\mbox{$^2$}$, and let $T_{01}$ be an arc in
$\stackrel{\ \circ}{J_1}$, $T_{02}$ be an arc in $\stackrel{\
\circ}{J_{-1}}$. Let $\{W_n: n \in \mathbb{N}_m \cup \{0\}\}$ be a
family of pairwise disjoint countable sets in $\stackrel{\
\circ}{J}\mbox{$^2$}-\bigcup\{T_{n1} \cup T_{n2}: n \in
\mathbb{N}_m\}$, in which each $W_n$ is dense in $J^2$. Then there
exists a homeomorphism $\psi: J^2 \rightarrow J^2$ such that, for
any $n \in \mathbb{N}_m \cup \{0\}$ and any $w \in W_n$, one has
$\omega_{\psi}(w)=T_{n1}$ and $\alpha_{\psi}(\omega)=T_{n2}$.
\end{thm}

\begin{proof}
Write $\mathbf{W}=\bigcup \{W_n: n \in \mathbb{N}_m \cup \{0\}\}$.
It is easy to show that there exist a family $\{D_{n1}, D_{n2}: n
\in \mathbb{N}_m \cup \{0\}\}$ of pairwise disjoint disc and a
family $\{L_{n1}, L_{n2}: n \in \mathbb{N}_m\}$ of pairwise disjoint
piecewise linear arcs such that

\medskip

(1)\ \ For $n \in \mathbb{N}_m \cup \{0\}$, $D_{n1} \in
\mathbf{X}(1, Ac)$, and $D_{n2} \in \mathbf{X}(-1, Ac)$;

\medskip

(2)\ \ $T_{01} \subset \mathrm{Int}(D_{01} \cup J_1)$, and $T_{02}
\subset \mathrm{Int}(D_{02} \cup J_{-1})$;

\medskip

(3)\ \ For $n \in \mathbb{N}_m$, $T_{n1} \subset \stackrel{\
\circ}{D_{n1}}$, and $T_{n2} \subset \stackrel{\ \circ}{D_{n2}}$;

\medskip

(4)\ \ For $n \in \mathbb{N}_m$,\; $L_{n1}$ is a cable from $T_{n1}$
to $J_1$,\; $D_{n1}$ is a hutch of $L_{n1} \cup T_{n1}$,\; and
$L_{n1} \cap \mathbf{W}=\emptyset$;

\medskip

(5)\ \ For $n \in \mathbb{N}_m$,\; $L_{n2}$ is a cable from $T_{n2}$
to $J_{-1}$,\; $D_{n2}$ is a hutch of $L_{n2} \cup T_{n2}$,\; and
$L_{n2} \cap \mathbf{W}=\emptyset$.

\medskip

For $n \in \mathbb{N}_m$, from Lemma \ref{lem:4-3} we see that
$L_{n1} \cup T_{n1}$ and $L_{n2} \cup T_{n2}$ are permeable hanging
sets. Of course, $T_{01}$ and $T_{02}$ are also permeable hanging
sets. By Theorem \ref{thm:4-7}, there exists a homeomorphism $\psi:
J^2 \rightarrow J^2$ having the properties described in Theorem
\ref{cor:4-10}.
\end{proof}

\subsection*{Acknowledgements}
Jiehua Mai and Fanping Zeng are supported by NNSF of China (Grant
No. 12261006) and Project of Guangxi First Class Disciplines of
Statistics (No. GJKY-2022-01); Enhui Shi is supported by NNSF of
China (Grant No. 12271388); Kesong Yan is supported by NNSF of China
(Grant No. 12171175).

\end{document}